\documentclass[11pt]{article}
\usepackage{amsfonts,amsmath,amssymb,dsfont,amscd,amsthm,xspace,mathrsfs}
\usepackage{fullpage}
\usepackage{graphicx}
\usepackage[all]{xy}
\usepackage{xcolor}
\usepackage{subfigure}

\newtheorem{theorem}{Theorem}
\newtheorem{lemma}[theorem]{Lemma}
\newtheorem{corollary}[theorem]{Corollary}
\newtheorem{conjecture}[theorem]{Conjecture}
\newtheorem{definition}[theorem]{Definition}

\newtheorem{remark}[theorem]{Remark}
\newtheorem{cons}{Construction}

\title{Acyclic and complete coloring of digraphs with  the minimum and maximum possible numbers of colors}
\author{Mika Olsen\footnotemark[1] \and Christian Rubio-Montiel\footnotemark[2] \and Alejandra Silva-Ram{\' i}rez\footnotemark[1]}

\begin{document}
\maketitle
\def\thefootnote{\fnsymbol{footnote}}
	\footnotetext[1]{Departamento de Matem{\' a}ticas Aplicadas y Sistemas, UAM-Cuajimalpa, Mexico City, Mexico. {\tt [olsen|alejandra.silva]@cua.uam.mx}.}
	\footnotetext[2]{Divisi{\' o}n de Matem{\' a}ticas e Ingenier{\' i}a, FES Acatl{\' a}n, Universidad Nacional Aut{\'o}noma de M{\' e}xico, Naucalpan, Mexico. {\tt christian.rubio@acatlan.unam.mx}.}
	
	
\begin{abstract} 
The dichromatic and diachromatic numbers of a digraph are the minimum and maximum numbers of colors, respectively, in acyclic and complete colorings of the digraph. 
In this paper, we construct, for all $r \leq t$, non-symmetric digraphs with dichromatic number $r$ and diachromatic number $t$. 
Moreover, we discuss the existence of asymmetric digraphs with dichromatic number \( r \) and diachromatic number \( t \geq r \), establishing a quadratic upper bound \( b(r) \leq t \).
%
\end{abstract}



\section{Introduction}

In graph theory, a classical problem is to determine, for two given parameters, whether there exists a (di)graph realizing every valid combination of these parameters. That is, given two graph invariants $\varphi$ and $\psi$, the goal is to characterize the pairs $(\alpha, \beta)$ for which there exists a (di)graph $G$ such that $\varphi(G) = \alpha$ and $\psi(G) = \beta$; see, for example, \cite{MR4923770,MR165515,Erdös_1959,MR3711038,MR4683677,MR1506881}. 
A well-known instance of this type of problem considers the minimum and maximum number of colors in a proper and complete vertex coloring: the \emph{chromatic number} $\chi(G)$ of a graph $G$ is the minimum number of colors in such a coloring, while the \emph{achromatic number} $\psi(G)$, introduced by Harary et al.\ in 1967~\cite{MR0272662}, is the maximum number of colors for such a coloring.
In 1979, Bhave~\cite{MR532949} proved that for every pair of integers \( r \leq t \), there exists a graph \( G \) such that \( \chi(G) = r \) and \( \psi(G) = t \). For further information, we refer the reader to the surveys~\cite{MR1477743,MR1427576}.

A natural question is whether this property can be extended to digraphs. We use the notion of \emph{acyclic colorings} introduced by Neumann-Lara~\cite{MR3875016} in 1982 as a natural generalization of proper colorings to digraphs, together with the extension of complete colorings proposed by Edwards~\cite{MR2998438}, where a coloring is called complete if, for every ordered pair of distinct colors $(i, j)$, there exists at least one arc directed from a vertex of color $i$ to a vertex of color $j$. Recently, in 2019, Araujo-Pardo et al.~\cite{MR3875016} formalized the \emph{achromatic number} for digraphs using acyclic colorings; this generalization applies to all digraphs and preserves the interpolation property known from the undirected case.

Following the terminology in~\cite{MR3875016,MR693366}, the \emph{dichromatic number} $\operatorname{dc}(D)$ of a digraph D is the minimum number of colors in an acyclic coloring, whereas the \emph{diachromatic number} $\operatorname{dac}(D)$ is the maximum number of colors in a complete acyclic coloring. In~\cite{MR4781357}, Zykov sums were used to construct digraphs realizing specific parameter pairs $(r, t)$ with $r \leq t$, such that $\operatorname{dc}(D) = r$ and $\operatorname{dac}(D) = t$.

This paper extends to the directed case the classical result asserting that for every pair of integers $a \leq b$, there exists a graph with chromatic number $a$ and achromatic number $b$. Since symmetric digraphs coincide with graphs in this context, their dichromatic and diachromatic numbers reduce to the corresponding undirected parameters.
We focus our attention on non-symmetric digraphs.  
First, we prove that for every pair of integers $r \leq t$, there exists a non-symmetric digraph $D$ such that $\operatorname{dc}(D) = r$ and $\operatorname{dac}(D) = t$. Second, we provide a quadratic upper bound $b(r)$ such that for every integer $t \geq b(r)$, there exists an \emph{asymmetric} digraph $D$ with $\operatorname{dc}(D) = r$ and $\operatorname{dac}(D) = t$. 
It is important to observe that these pairs improve the existing known pairs established in \cite{MR4781357}.


\section{Definitions and preliminary results}

In this paper, we consider finite digraphs. Let $k$ be a non-negative integer. An \emph{acyclic coloring} of a digraph $D$ is a function $\varphi : V(D)\to [k]$ with $[k]:=\{1,2,\ldots,k\}$. A \emph{chromatic class} of color $i$ is denoted by $C_i=\{v\in V(D) \mid \varphi(v) = i \}$, for $i\in [k]$. A \emph{sink} (resp. \emph{source}), in a digraph is a vertex with no outgoing arcs (resp. no incoming arcs).
The coloring $\varphi$ is \emph{acyclic} if no  chromatic class induces a subdigraph with a directed cycle. The \emph{dichromatic number} $dc(D)$ and the \emph{diachromatic number} $dac(D)$ are the smallest and the largest $k$ such that $D$ admits an acyclic coloring, respectively. An acyclic and complete coloring  that uses  $dc(D)$ (resp. $dac=t$) colors, is an \emph{optimal coloring} for the dichromatic number (resp. diachromatic number).
For general concepts we refer the reader to \cite{MR2472389} and  Chapter $12$ in \cite{MR2450569}.
	
For two nonempty  of vertex sets $X,Y$ of a digraph $D$, we define 
\[[X, Y] = \{ (x,y)\in A(D) \mid  x \in X, y \in Y \}.\]
	
Let $\mathds{Z}_m$ be the cyclic group of integers modulo  $m$, where $m\in \mathds{N}$ and  $J$ is a nonempty subset of  $\mathds{Z}_m \setminus \{0\}$ such that  $w\in J$ if and only  if $-w\notin J$ for every $w\in \mathds{Z}_m \setminus \{0\}$.
The \emph{circulant digraph} $\overrightarrow{C}_m(J)$ is defined by 
\[V(\overrightarrow{C}_m(J))=\mathds{Z}_m \mbox{ and  }A(\overrightarrow{C}_m(J))=\{(i,j) : i,j \in \mathds{Z}_m\mbox{ and } j-i \in J\}.\]
	
The following results are already known, and we will use them throughout this paper.

\begin{theorem}\cite{MR737621}\label{teo1}
If $T$ is a regular tournament with $2n + 1$ vertices, then $dc(T) = 2$ if and only if $T \cong \overrightarrow {C}_{2n+1}(1, \ldots, n) $.
\end{theorem}

\begin{theorem}\cite{MR3875016}\label{teo2}
Let $T$ be a circulant tournament or a transitive tournament of order n. Then 
\[dac (T)=\left\lceil\frac{n}{2}\right\rceil.\]
\end{theorem}

The next corollary is a consequence of Theorems \ref{teo1} and \ref{teo2}.

\begin{corollary}
For every $t\geq 2$, there exists a  $2$-dichromatic tournament with $dac(T)=t$.
\end{corollary}

\begin{definition}\cite{MR2564801}
For every pair of integers $r,s$ with $r\geq 3, s\geq 0$; let  $H_{p,s} =\overrightarrow {C}_{(r-1)(p+1)+1}(J)$ for $p=(r-2)(s+1)+1$ and $$J=\{1,2,\ldots,p\}\cup\{p+2,p+3,\ldots,2p-s\}\cup\ldots\cup\{(r-2)p+r-1\}.$$
\end{definition}

When $s=0$, we denote $H_{p,s}$ by $H_r$.

\begin{theorem}\label{teo5}\cite{MR2564801}
For $r\geq 3, s\geq 0$,  $H_{p,s}$ is a vertex critical $r$-dichromatic circulant tournament.
\end{theorem}

The next corollary is a consequence  of Theorems \ref{teo2} and \ref{teo5}.

\begin{corollary}\label{cor6}
For every $r\geq3$ and $s\geq 0$, there exists a  circulant tournament $T$ such that   $dc(T)=r$ and
$$dac(T)= \frac{(r^2-3r+2)s}{2}+\frac{r^2-r+2}{2}. $$	
		
When $s=0$, there exists a circulant tournament $T$ with $dc(T)=r$ and   $dac(T)= (r^2-r+2)/2.$ 
\end{corollary}

Table \ref{table1} shows the values for the diachromatic number in terms of the dichromatic number, obtained in the Corollary \ref{cor6}.
Note the gaps between the values of the diachromatic number depend on the  dichromatic number.

\begin{table}[htbp!]
	\centering
	\begin{tabular}{|c|c|c|c|c|c|c|c|c|c|c|}
		\hline
		$dc(T)$  & $dac(T)$ \\ \hline
		3 & 4,5,6,7,8,9\ldots    \\ \hline
		4&7,10,13,16,19,22\ldots \\ \hline
		5 & 11,17,23,29,35,41\ldots  \\  \hline
		&\vdots
		\\ \hline
		9 &  37,65,93,121,149,177\ldots\\ \hline
		&\vdots
		\\ \hline
	\end{tabular}
	\caption{The $dc(T) $ and the $dac(T)$ of the  family in  Corollary \ref{cor6}.}\label{table1}
\end{table}

The next result is a reformulation of the Interpolation Theorem in terms of complete coloring.

\begin{theorem}\label{teo7}\cite{MR3875016}
Let $D$ be a digraph. For every integer $l$ with $dc(D) \leq l \leq dac(D)$ there exists a  complete and acyclic coloring with $l$ colors.
\end{theorem}


\section{Construction}
	
In this section, we define a construction that allows us to fill the gaps observed in Table \ref{table1}.

\begin{cons}\label{cons8}
Let $D$ be a digraph and let $P=\{v_1,v_2,\ldots,v_s\}$ be a path such that $V(D)\cap V(P)= \emptyset$.
For a non-negative integer  $n$ and  $v_0\in V(D)$,  we define the digraph $D_n(v_0)$ as follows:
	$$\begin{array}{ccl}
	V(D_n(v_0)) &=& V(D)\cup V(P),\\\\ 
	A(D_n(v_0))& =& A(D)\cup A(P')\cup M,
	\end{array}$$

where $V(P^{'})=\{v_0,v_1,v_2,\ldots, v_s\}$ and the set of arcs of $P'$ is given by 
		
	\begin{center}
		$A(P') = \left \{ \begin{array}{cll} (v_j,v_i) & \mbox{ if }j-i \mbox{ is even;}
		\\ (v_i,v_j) &\mbox{ if }j-i \mbox{ is odd}. \end{array}\right.$ 
	\end{center}

The arcs between  vertices of the digraph $D$ and  vertices of the path $P$ are defined by  the  set $M$. For all $x\in V(D)-\{v_0\} $ and $1\leq {2i} \leq {2i+1} \leq s $ 

	\begin{center}
		$M = \left \{ \begin{array}{cl} (x,v_{2i}),(v_{2i+1}, x )& \mbox{ if } x\in N^{-}(v_0)\setminus N^{+}(v_0);
		\\(v_{2i}, x),(x, v_{2i+1}) & \mbox{ if } x\in N^{+}(v_0).
		\end{array}\right.$
	\end{center}

\end{cons}

Note that the arc $(v_0, x)$ may be symmetric. Furthermore $D_0(v_0)=D$.

In Figure \ref{fig1}, observe that if $i<j$ are both even (odd), then the  arc between $v_i$ and $v_j$ is $v_jv_i$, so we have the following remake. 
\begin{figure}[htbp!]	
	\centering
	\includegraphics[scale=0.5]{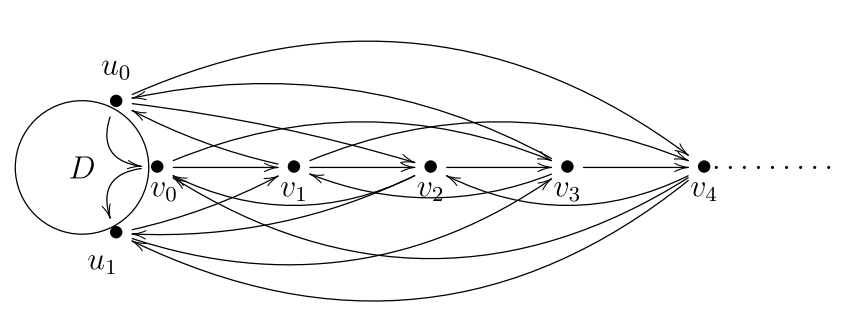}
	\caption{The arcs between the vertices of  digraph $D$  and vertices of path $P$ in the Construction  \ref{cons8}.}\label{fig1}
\end{figure}
\begin{remark}\label{rem9}
Consider the path $P'$ in Construction \ref{cons8} and let $1\leq {2i} < {2i-1} \leq s$. Then 
	\begin{enumerate}
		\item [(i)] $P^{'} \left\langle \{v_{2i}\}_{i=0,1,\ldots,\left\lfloor \frac{s}{2}\right\rfloor} \right\rangle $ is a transitive subtournament  with $v_0$ is a sink.
		\item [(ii)] $P^{'}\left\langle \{v_{2i-1}\}_{i=1,\ldots,\left\lfloor \frac{s}{2}\right\rfloor}  \right\rangle $ is a transitive subtournament with $v_1$ is a sink.
	\end{enumerate}
\end{remark}


The following lemma determines the diachromatic number of a digraph  obtained by Construction \ref{cons8},  in terms of  dichromatic number and the length of path $P$ when the length is even.
	
\begin{lemma}\label{lem10}
Let $D$ be a digraph with $dac(D)=t$ and let $k$  be a non-negative integer. Then $$dac(D_{2k}(v_0))= t+k.$$
\end{lemma}
\begin{proof}
Let $D_{2k}(v_0)$ be a digraph obtained by  Construction \ref{cons8}, and consider an optimal coloring $\varphi:V(D) \to [t]$ for the diachromatic number.
We extend  the  coloring  $\varphi$ to a coloring $\varphi':V(D_{2k}(v_0))\to [t+k]$ as follows  
		
\[\varphi'(w) =  \left \{ \begin{array}{lll}
\varphi(w)& \mbox{if  }w\in V(D);\\
t+i & w\in\{v_{2i-1},v_{2i} \}  \mbox { for all } i\in[k].
\end{array}\right.\]

The chromatic classes  of $\varphi'$ are 
\[C_{i}'=C_{i} \mbox{ for } i\in[t] \mbox{ and } C_{t+i}'= \{v_{2i-1}, v_{2i}\} \mbox{ for } i\in [k].\]

Note that  for $i\in [t]$ the chromatic  classes  $C_i'$ are acyclic  by construction and  for $i\in [k]$ the   $C_{t+i}'$ are acyclic classes of order $2$.  Therefore coloring $\varphi'$ is acyclic.
		
In order to prove that $\varphi'$ is complete, we consider the following the three cases  separately: $i< j$ for $i,j\in [t]$; $i\in [t]$, $t<j\leq t+k$; and $t<i<j\leq t+k$.
Let  $i,j\in [t]$ with  $i<j$ and $u,v \in V(D_{2k}(v_0))$ such that $\varphi'(u)=i$ and $\varphi'(v)=j$. By  definition of  the coloring $\varphi'$,   $v, u$  are vertices that belongs to digraph $D$. Since $u,v \in V(D)$, by definition of coloring $\varphi$ it follows that $[C_i',C_j'],[C_j',C_i']\neq \emptyset$.
	Let $i\in [t]$, $j=t+l$ with $l\in [k]$ and $u,v \in V(D_{2k}(v_0))$ such that $\varphi'(u)=i$ and $\varphi'(v)=j$. By definition of the chromatic class,  $v\in  \{v_{2l-1}, v_{2l}\}$, and by the set of arc $M$ in  Construction \ref{cons8},  it follows that  $[C_{i}',C_{j}' ],[C_{j}',C_{i}']\neq \emptyset$.  Let $t<i<j\leq t+k$ and  $u,v \in V(D_{2k}(v_0))$ such that $\varphi'(u)=i$ and $\varphi'(v)=j$. By definition of the chromatic classes and  of the adjacencies on the path $P$ in  Construction \ref{cons8}, it follows that $[C_{j}',C_{l}'], [C_{l}',C_{j}']\neq \emptyset$. Hence, $\varphi'$ is  complete and $dac(D_{2k}(v_0))\geq t+k$.
		
Assume, to the contrary, that $dac(D_{2k}(v_0))\geq  t+k+1$. By  Theorem \ref{teo7},  there exists a complete and acyclic coloring  $\phi:V(D_{2k}(v_0))\to [t+k+1].$ Since $\phi$ restricted to $V(D)$ uses at most $t$ colors, the coloring $\phi$ restricted to $V(P)$ uses at list  the remaining $k+1$ colors. By  Pigeonhole Principle, there is a singular  chromatic class $C_\alpha' =\{v_j\}$ with $v_j \in V(P)$.
If  $j=2i$ for $i\in[k]$, by  definition  the set of arc $M$ in Construction \ref{cons8}, it follows that $[C_{j}', C_{\alpha}']=\emptyset$. Analogously, if $j=2i-1$  for $i\in[k]$ then  $[C_{\alpha}', C_{j}']=\emptyset$.  Both cases contradicts that the coloring $\phi$ is complete, thus  $dac(D_{2k}(v_0))\leq n+k$ and  the result follows.
\end{proof}

Let $\overleftrightarrow {K}_n$ be complete symmetric digraph with $n$ vertices. Recall that, for a  complete symmetric digraph $\overleftrightarrow{K}_n$ the dichromatic number and diachromatic number is equal to $n$. 

\begin{theorem}\label{teo11}
For each pair of integers $r,t$ with $ 1\leq r\leq t$, there exists a non-symmetric $D$ digraph, such that $dc(D)=r$ and $dac(D)=t.$
\end{theorem}
\begin{proof}
Note that for $n\geq 1$, $\overleftrightarrow{K}_{n,1}(v_0)$ is a non-symmetric digraph such that $v_1$ is a source in  $V(\overleftrightarrow{K}_n)$, then 
$ dac (\overleftrightarrow{K}_{n,1} (v_0))=dac (\overleftrightarrow{K}_n)=n$. 
		
Let $n\geq 2$, $k\geq 1$. Consider the digraph  $\overleftrightarrow{K}_{n,2k}(v_0)$ obtained using  Construction \ref{cons8}. By Lemma \ref{lem10},  $dac(\overleftrightarrow{K}_{n,2k}(v_0))=n+k.$ 
		
Let $\varphi:V(\overleftrightarrow { K}_n) \to [n]$  be an optimal coloring for number dichromatic, such that $\varphi(v_0)=2$.
Since $\overleftrightarrow{K}_n$ is  an induced subdigraph of $\overleftrightarrow{K}_{n,2k}(v_0)$, it follows that  $dc(\overleftrightarrow{K}_{n,2k}(v_0))\geq n$. We extend the coloring $\varphi$  to a coloring of  $\varphi':V(\overleftrightarrow{K}_{n,2k}(v_0))\to [n]$ as follows:

	\begin{center}
		$\varphi'(w) = \left \{ \begin{array}{lll}
		\varphi(w)& \mbox{ if }w\in V(\overleftrightarrow{ K}_n);\\
		\varphi(v_{2i-1})=1&\mbox{ for } i\in [k] ;\\
		\varphi(v_{2i})=2&\mbox{ for } i\in[k].
		\end{array}\right.$
	\end{center}
		
Let $u \in V(\overleftrightarrow{K}_n)$ such that $\varphi'(u)=1$ and  $C_1', C_2', \ldots, C_n'$ are the chromatic classes of the coloring $\varphi'$. Then $C_{1}'=\{u\}\cup \{v_1, v_3,\ldots, v_{2k-1}\}$, $C_{2}'=\{v_0, v_2,\ldots, v_{2k}\}$ and $C_i^{'}=C_i$ for $i\geq 3$. Note that the  chromatic class $C_2'$ is acyclic by Remark \ref{rem9} $(i)$. By  definition of the arc  set $M$ in Construction \ref{cons8},  $u$ is a source in  chromatic class $C_1'$ and using  Remark  \ref{rem9} $(ii)$  we have that  $C_1'$ is acyclic.
Therefore, the coloring $\varphi'$ is acyclic,   $dc(\overleftrightarrow{K}_{n, 2k}(v_0))=n$ and the proof is complete.
\end{proof}

Observe that every digraph obtained in Theorem \ref{teo11} has symmetric arcs except when  $r=1$.
In the following theorem, we consider asymmetric digraphs.
	
\begin{theorem}\label{teo12}	
Let $r\geq 2$. Let $D$ be an $r$-dichromatic digraph admitting 	an  optimal coloring with a singular class $\{u\}$. Then $dc(D_n(v_0))=r$ for   $n\geq 0 $ and $v_0\neq u$. 
\end{theorem}
\begin{proof}
For $r\geq 2$, $n\geq 0$ and $v_0\neq u$, let  $D_n(v_0)$  be  a digraph obtained  by Construction  \ref{cons8} and  let $\varphi:V(D_n(v_0)) \to [r]$ be  an  optimal  coloring for number dichromatic with a  singular  chromatic class  $\{u\}$.  Assume  that $C_1=\{u\}$ and $v_0 \in C_2$. We extend  the optimal coloring $\varphi$ to a coloring $\varphi' :V(D_n(v_0)) \to [r]$ as follows
	\begin{center}
		$\varphi'(w) = \left \{ \begin{array}{cll} 1 &\mbox{ if } w=v_{2i-1}, &\mbox{ with } i=1,2,\ldots, \left\lceil \frac{n}{2} \right\rceil;\\
		\\2 &\mbox{ if } w=v_{2i}, &\mbox{ with } i=0,1,\ldots,\left\lfloor \frac{n}{2}\right\rfloor;\\
		\\\varphi(w)& \mbox{ if  }w\in V(D).
		\end{array}\right.$
	\end{center}
		
The  chromatic classes of  $\varphi'$ are
\[C_1'=C_1 \cup \{v_1,v_3, \ldots, v_{2\left\lceil\frac{n}{2}\right\rceil -1 }\}, \quad  C_2^{'}=C_2\cup \{v_2,v_4,\ldots,v_{2\left \lfloor \frac{n}{2} \right\rfloor}\}\textrm{ and }C_i^{'}=C_i \mbox{ if } i\geq 3.\]

Note that every chromatic class $C'_i$,  with $i\geq 3$,  is  acyclic by construction. The  chromatic class  $C_{1}'\setminus \{u\}$ is acyclic by Remark $(i)$. By definition of the arc set $M$ in Construction \ref{cons8},  if $(u,v_0)\in A(D)$ then   $(v_{2i-1} , u)\in A(D_n(v_0))$  for  $i\in\{1,2,\ldots,\left \lceil \frac{n}{2} \right\rceil\}$; and  if $( v_0, u)\in A(D)$ then   $(u, v_{2i-1} )\in A(D_n(v_0))$  for  $i\in\{1,2,\ldots,\left \lceil \frac{n}{2} \right\rceil \}$. In both cases we have a transitive subtournament, therefore the chromatic class $C_{1}'$ is acyclic.
		
In order to prove that  $C_{2}'$ is acyclic, suppose for contradiction,  that there is  a cycle of color $2$  in $D_n(v_ 0)$. Since $C_2$ is acyclic by construction, the cycle contains a vertex $v_{2i}$ for  some $i\in\{1,2,\ldots,\left \lfloor \frac{n}{2} \right\rfloor\}$.
Let $\mathcal{C}$ be a cycle   of color $2$ with the minimum number of vertices of $P$.
If $v_0\notin V(\mathcal{C})$, consider minimum integer $\alpha$ such that
$v_{2\alpha}\in V(P)\cap V(\mathcal{C})$. In this case $N_{C_2'}^+(v_{2\alpha})=N_{C_2'}^+(v_0)$ and $N_{C_2'}^-(v_{2\alpha})=N_{C_2'}^-(v_0)$ and we can construct a cycle $\mathcal{C}'$ with less vertices of $P$ exchanging the vertices $v_ 0$ and $v_{2\alpha}$, see figure  \ref{fig2}(a),  contradicting the choice of $\mathcal{C}$. Thus $v_0 \in V(\mathcal{C}).$ 
\begin{figure}[ht!]	
\centering
\subfigure[Exchanging $v_0$ with $v_{2\alpha}.$]{
\includegraphics[width=58mm]{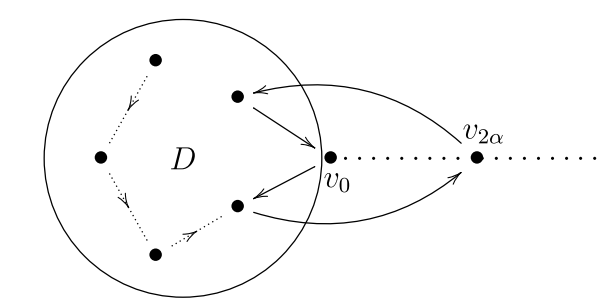}}\hspace{1cm}
\subfigure[The cycle $\mathscr{C}'.$]{
\includegraphics[width=62mm]{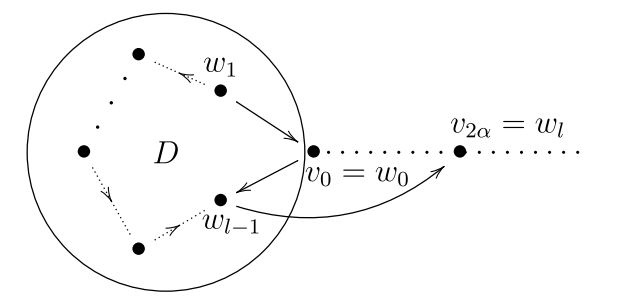}}
\caption{The cycle $\mathscr{C}$ in the  chromatic class of color $2.$ }
\label{fig2}
\end{figure}


Let $\mathcal{C}=(w_0,w_1,\ldots, w_k,w_0)$ such that $v_0=w_0$   and let   $l\in [k]$ be the minimum integer such that and   $w_l=v_{2\alpha}$ for $\alpha>0$.  By construction,  we have that $v_{2\alpha}\neq w_1$ and so  $w_1\in V(D)$. By construction
$N_{D}^+(v_{2\alpha})=N_{D}^+(v_0)$ and $N_{D}^-(v_{2\alpha})=N_{D}^-(v_0)$, since  $w_1\in N^-(w_2 )$ and $w_1\in N^+(w_0)$ then $v_{2\alpha}\neq w_2$. Thus $l\geq3$ and  $w_i\in V(D)$ for $i<l$.
Since  $(w_{l-1}, v_{2\alpha})\in A(D)$ and by  definition of the arc set $M$, is follows that  $(w_{l-1}, w _0)\in A(D)$.
Consider the  cycle $\mathcal{C}'=(w_0, w_1,\ldots,w_{j-1},w_0)$, see figure \ref{fig2}(b). All the vertices of $\mathcal{C}'$ belongs to the chromatic class $C_2$,  contradiction the  $C_2$ is acyclic. Thus, there are no monochromatic cycles of color $2$ and the result follows.
\end{proof}

By Lemma \ref{lem10} and Theorem \ref{teo12}, we have the following.

\begin{theorem}\label{teo13}
Let $D$ be a  critical $r$-dichromatic digraph with $dac(D)=t_0$.  Then for all $t \geq t_0$ there exists an asymmetric digraph $D'$ such that $dc(D')=r$ and $dac(D')=t$.
\end{theorem}

The next  theorem is a consequence of  Theorems \ref{teo1}, \ref{teo2},  \ref{teo12} and  Corollary \ref{cor6}.

\begin{theorem}\label{teo14}
For every  integer $r\geq 2$ and $t\geq ({r^2-r+2})/{2}$ ,  there exists an  asymmetric digraph $D$  with $dc(D)=r$ and $dac(D)=t.$
\end{theorem}
	
Note that $b(r)\leq ({r^2-r+2})/{2}$   by  Theorem \ref{teo14}. However, for some values of $b(r)$,  we can improve the upper bound using the following theorems.

\begin{theorem}\label{teo15}\cite{MR1310883}
$ST_{11}$ is the unique $4$-dichromatic regular tournament of order $11$.
\end{theorem}

In Section {\bf 4.6 An Application} \cite{MR1817491}, V. Neumann-Lara constructed a $5$-dichromatic tournament of order $19$ and recently Bellitto et al. proved the following.     
\begin{theorem}\label{teo16}\cite{5dicromatico}
The minimum order of a $5$-dichromatic tournament is $19$.
\end{theorem}

\section{Discussion}
In this paper, we introduced two distinct constructions that contribute to the understanding of the relationship between the dichromatic and diachromatic numbers of digraphs. The first construction establishes that, for every pair of integers $2 \leq r \leq t$, there exist non-symmetric digraphs $D$ satisfying $dc(D) = r$ and $dac(D) = t$. The second construction provides an infinite family of integer pairs $(r, t)$, with $r \geq 2$ and $t \geq \frac{r^2 - r + 2}{2}$, for which there exist asymmetric digraphs such that $dc(D) = r$ and $dac(D) = t$.

Table~\ref{table2} summarizes the best known upper bounds $b(r)$ for the diachromatic number in terms of the dichromatic number $r$. Note that the bounds for $r = 4$ and $r = 5$ are improved by Theorems~\ref{teo15} and~\ref{teo16}, respectively.

\begin{table}[htbp!]
    \centering
    \begin{tabular}{|c|c|c|c|c|c|c|c|c|c|c|}
        \hline
        \( r \)  & 2 & 3 & 4 & 5 & 6 & 7 & 8 & 9 & 10 \\ \hline
        \( b(r) \) & 2 & 4 & 6 & 10 & 16 & 22 & 29 & 37 & 46 \\ \hline  
    \end{tabular}
    \caption{Upper bounds \( b(r) \) for the diachromatic number in terms of the dichromatic number \( r \).}
    \label{table2}
\end{table}

Despite these contributions, the general problem remains unresolved: Given a pair \( (r, t) \) with \( t \geq r \), does there always exist an asymmetric digraph \( D \) such that \( dc(D) = r \) and \( dac(D) = t \)? This question remains particularly compelling in the case \( r = t > 2 \), for which no explicit construction is currently known. In this direction, we propose the following conjecture:

\begin{conjecture}
    There are no asymmetric digraphs \( D \) such that \( dc(D) = dac(D) = r \) for \( r > 2 \).
\end{conjecture}

A complete characterization of the attainable pairs \( (dc(D), dac(D)) \) is still lacking. Addressing this open problem will require the development of new techniques either to improve the existing upper bounds $b(r)$, to construct suitable asymmetric digraphs, or to prove 
that such digraphs do not exist.


\bibliographystyle{plain}
\bibliography{biblio}
	
\end{document}